\documentclass{article}
\usepackage[utf8]{inputenc}
\usepackage[title]{appendix}
\usepackage{bbm}
\usepackage{amssymb}
\usepackage[english]{babel}    
\usepackage{authblk}
\usepackage{amsmath,amsthm}           
\usepackage[utf8]{inputenc}    
\usepackage[T1]{fontenc}       
\usepackage{longtable}         
\usepackage{exscale}           
\usepackage[final]{graphicx}   
\usepackage[sort]{cite}        
\usepackage{array}             
\usepackage{eucal} 
\usepackage{hyperref}
\usepackage{wasysym}           
\usepackage[a4paper]{geometry} 
\usepackage[multiuser]{fixme}  
\usepackage{xspace}            
\usepackage{tikz}              
\usepackage{ifdraft}           
\usepackage[expansion=false    
           ]{microtype}        
\usepackage[nottoc]{tocbibind} 
\usepackage[all]{xy}
\usepackage{slashed}
\usepackage{cleveref}
\usetikzlibrary{matrix}
\usepackage{tikz-cd}
\usepackage{amsthm}

\newtheorem{definition}{Definition}[section]
\newtheorem{lemma}{Lemma}[section]

\newtheorem{proposition}{Proposition}[section]
\newtheorem{corollary}{Corollary}[section]

\newtheorem{theorem}{Theorem}[section]
\newtheorem{convention}{Convention}[section]
\newtheorem{remark}{Remark}[section]

\renewcommand{\P}{\mathcal{P}}
\newcommand{\Pu}{u\mathcal{P}}
\newcommand{\C}{\mathcal{C}}
\newcommand{\Cu}{u\mathcal{C}}
\newcommand{\E}{\mathcal{E}}
\newcommand{\Q}{\mathbb{Q}}
\newcommand{\R}{\mathbb{R}}
\newcommand{\RR}{\mathcal{R}}

\renewcommand{\1}{\mathbbm{1}}
\newcommand{\Ass}{\mathcal{ASS}}
\newcommand{\Com}{\mathcal{COM}}

\newcommand{\BE}{\mathcal{BE}}
\renewcommand{\S}{\mathbf{S}}
\newcommand{\Cocom}{\mathcal{COCOM}}
\DeclareMathOperator{\Ind}{Ind}
\DeclareMathOperator{\End}{End}

\newcommand{\N}{\mathbb{N}}
\DeclareMathOperator{\id}{id}

\DeclareMathOperator{\MC}{MC}

\title{
A Lie theoretic approach to the twisting procedure and Maurer-Cartan simplicial sets over arbitrary rings}
\author[1]{Niek de Kleijn}
\author[2]{Felix Wierstra}
\date{}
\affil[1]{Delft Institute of Applied Mathematics, Technical University Delft, Mekelweg 4, Delft, The Netherlands, n.dekleijn@tudelft.nl}
\affil[2]{Korteweg-de Vries Institute for Mathematics, University of Amsterdam, Science Park 105-107, Amsterdam, The Netherlands, felix.wierstra@gmail.com}

\begin{document}
\maketitle

\begin{abstract}
The Deligne-Getzler-Hinich--$\infty$-groupoid or Maurer-Cartan simplicial set of an $L_\infty$-alge\-bra plays an important role in deformation theory and many other areas of mathematics. Unfortunately, this construction only works over a field of characteristic $0$.  The goal of this paper is to show that the notions of Maurer-Cartan equation and Maurer-Cartan simplicial set can be defined for a much larger number of operads than just the $L_\infty$-operad. More precisely, we show that the Koszul dual of every unital Hopf cooperad (a cooperad in the category of unital associative algebras) with an arity $0$ operation admits a twisting procedure, a natural notion of Maurer-Cartan equation and under some mild additional assumptions can also be integrated to a Maurer-Cartan simplicial set. In particular, we show that the Koszul dual of the Barratt-Eccles operad and its $\E_n$-suboperads admit Maurer-Cartan simplicial sets. In this paper, we will work over arbitrary rings.
\end{abstract}

\section{Introduction}

Lie algebras up to homotopy, better known as $L_\infty$-algebras, play an important role in many areas of mathematics like deformation quantization and deformation theory, mathematical physics, symplectic geometry, rational homotopy theory and many others.   It is a well known philosophy that, over a field of characteristic $0$, all deformation problems are controlled by the Maurer-Cartan elements in an $L_\infty$-algebra. This philosophy goes back to Deligne, Drinfeld, Feigin, Hinich, Kontsevich-Soibelman, Manetti, and many others, and was made precise by Lurie in \cite{Lurie2010} and Pridham in \cite{Pridham}. The information about such a deformation problem can be conveniently organized in a simplicial set, which is known as the Deligne-Getzler-Hinich--$\infty$-groupoid, Maurer-Cartan simplicial set or nerve. This simplicial set was first defined by Getzler in the case of $L_\infty$-algebra in \cite{Getzler}. 

When working over fields of general characteristic, the situation is a lot more complicated. Deformation problems or formal moduli problems are no longer controlled by $L_\infty$-algebras but by $E_n$-algebras. More explicitly, Lurie showed that for $n< \infty$, there is an equivalence between the  $\infty$-category of of formal $E_n$-moduli problems and the $\infty$-category of augmented $E_n$-algebras. This was later extended to $n=\infty$ by Brantner and Mathew \cite{BrantnerMathew} and Brantner, Campos and Nuiten \cite{BrantnerCamposNuiten}.

Although all these approaches are great from a theoretical perspective, there has not been an analog of the Deligne-Hinich-Getzler--$\infty$-groupoid. With this we mean a simplicial set which encodes the homotopical information of the deformation problem, in which the simplices (Maurer-Cartan elements) are defined by explicit equations and operations like the twist are defined by explicit formulas. In \cite{deKleijnWierstra}, we made a first step towards this by showing that $A_\infty$-algebras (also known as $E_1$-algebras) admit a Maurer-Cartan simplicial set. In this paper, we show that a much larger number of operads have a natural notion of Maurer-Cartan equation, twisting procedure and Maurer-Cartan simplicial set. 

\medskip

To construct our Maurer-Cartan simplicial set over a general ring $R$, we use Kontsevich's perspective of formal pointed manifolds from \cite{Kontsevich} and generalize it to what could be seen as "formal pointed Lie groups". Over a field of characteristic $0$, Kontsevich defined an $L_\infty$-algebra on a graded vector space $V$ as a degree $-1$ coderivation  $Q:C(V)\rightarrow C(V)$ which squares to zero, where $C(V)$ denotes the cofree conilpotent cocommutative coalgebra cogenerated by $V$. The coderivation $Q$ is seen as a "vector field" on $\C(V)$. In this paper, we generalize this idea by taking more general cooperads $\C$ instead of the cocommutative cooperad. If we furthermore assume that $\C$ is a unital Hopf cooperad (a cooperad in the category of unital associative algebras), which admits an arity zero operation encoding the counit, then we define an $\Omega \C$-algebra on a graded $R$-module $V$ as a square-zero coderivation $Q:\C(V)\rightarrow \C(V)$ on the cofree conilpotent $\C$-coalgebra cogenerated by $V$. Here, $\Omega \C$ denotes the operadic cobar construction on $\C$ and the correspondence between $\Omega \C$-coalgebras and coderivations follows from the Rosetta Stone Theorem (see \cite{LV}, Theorem 10.1.13). To ensure that certain infinite sums converge, we need to introduce filtrations and restrict ourselves to complete or pro-nilpotent $\Omega \C$-algebras. To ensure that the theory works over arbitrary rings, we further need to assume that in arity $r$, $\C(r)$ is free as an $R[\S_r]$-module, where $\S_r$ denotes the symmetric group on $r$ elements.

The class of operads that satisfy these assumptions is quite large. For example, it contains  the normalized cochains on any operad $\P$ in simplicial sets, such that $\P(0)=*$,  $\P$ has a free symmetric group action in each arity and finitely many non-degenerate simplices in each arity. This is because the cochains on a topological operad have the cochain level cup product making it a unital Hopf cooperad. Some particular examples of such operads are the commutative operad, the associative operad, the Barratt-Eccles operad and the $\E_n$-suboperads of the Barratt-Eccles operad.

By a theorem of Moerdijk \cite{Moerdijk}, we can equip $\C(V)$ with an associative product 
\[
\star:\C(V)\otimes \C(V)\rightarrow \C(V)
\]
which is a generalization of the shuffle product on the cofree conilpotent cocommutative coalgebra and is therefore called the generalized shuffle product. This product turns $\C(V)$ into  a $\C$-$\Ass$-bialgebra, which can be interpreted as the analog of a Lie group, where $\star$ corresponds to the group structure and the differential $Q:\C(V)\rightarrow \C(V)$ corresponds to the manifold structure. The chain complex $V$ corresponds in this case to the tangent space or "Lie algebra".

This "Lie group" has many similarities to the theory of Lie groups and we generalize certain concepts from classical Lie theory to this new setting.  We start by defining the analog of the exponential map, which is a map 
\[
\exp:V_0\rightarrow \C(V)
\]
from the "Lie algebra" $V$ to the "Lie group" $\C(V)$ which shares many properties with the ordinary exponential map. This exponential map allows us to define a twisting procedure which twists the differential $Q:\C(V)\rightarrow \C(V)$ by a degree $0$ element $v\in V_0$, in a way that is analogous to the adjoint representation of a classical Lie group. More precisely, if $x\in \C(V)$  then the twisted differential $Q^v:\C(V) \rightarrow \C(V)$ is defined as
\[
Q^v:=\exp(-v)\star Q(\exp(v)\star x).
\]
The Maurer-Cartan equation then naturally follows as a flatness condition for this twist. This further gives an alternative to the twisting procedure from \cite{DSV}, which works for differential graded operads and over arbitrary rings instead of fields of characteristic $0$. 


If the cooperad $\C$ comes equipped with a map $\E_\infty\rightarrow \C$, where $\E_\infty$ denotes the cochains on the Barratt-Eccles operad $\E_\infty$, we can integrate every $\Omega \C$-algebra to a simplicial set which we call the Maurer-Cartan simplicial set. The main examples which satisfy this requirement are the Barratt-Eccles cooperad and its $E_n$-subcooperads. We finish by showing that this simplicial set is Kan complex and therefore an $\infty$-groupoid.



\subsection*{Acknowledgements}

The authors would like to thank Daniel Robert-Nicoud, Martin Markl, Igor Khavkine, Dion Leijnse, Noah Olander, Sergey Shadrin and Jos\'e Moreno-Fern\'andez for useful conversations and comments. The second author is supported by Dutch Research Organisation (NWO) grant number VI.Veni.202.046.

\section{Preliminaries and conventions}

In this section, we recall the necessary preliminaries and establish our notation and conventions.

\subsection{Conventions and notation}

In this paper, we work  in the category of chain complexes over a general ring $R$ with unit. We use a homological grading, so the differentials are of degree $-1$. Unless stated otherwise, all tensor products are taken over $R$ and we implicitly assume the Koszul sign rule, i.e. if $V$ and $W$ are two chain complexes then the switch map $V\otimes W \rightarrow W \otimes V$ induces an additional sign. The chain complex of $R$-linear maps between two chain complexes $V$ and $W$ is denoted by $\hom_R(V,W)$, often we drop the $R$ and write just $\hom(V,W)$. 

Let $G$  be a finite group and $V$  a chain complex with a linear $G$-action, then we denote the coinvariants with respect to the $G$ action by $V_G$. The invariants with respect to the $G$-action are denoted by $V^G$. The symmetric group in $r$ elements is denoted by $\S_r$, by convention we say that both $\S_1$ and $\S_0$ are the trivial group.

\subsection{Conventions on operads and cooperads}

We assume that the reader is familiar with the basic theory of algebraic operads and cooperads and otherwise refer the reader to \cite{Fresse00}, \cite{Fresse04} and \cite{LV}. Although the book \cite{LV} is written in characteristic $0$, most of the results hold over arbitrary rings with only small modifications which are explained later in this paper. 

Unless stated otherwise, all operads are in chain complexes over a ring $R$ (except for Section \ref{sec:examples} where we also consider operads in simplicial sets). Since we work over a general ring $R$, and not a field of characteristic $0$, there are a few special conventions we need in this paper that are not super standard and we recall those here. For more details, we mainly refer to \cite{Fresse04}, which is one of the few references that works in this generality.

The unit operad and the unit cooperad are both denoted by $I$ and defined as $I(r)=0$ if $r\neq1$ and  $I(1)=R$, with the appropriate (co)operad structure. From the context it should be clear whether we mean the operad or cooperad. An operad $\P$ is called augmented if there is an additional map $\epsilon:\P\rightarrow I$ which is a splitting for the operadic unit. A cooperad is coaugmented if there is a splitting $\eta:I\rightarrow \C$ for the cooperadic counit.

An operad $\P$ (resp. cooperad $\C$) is called unitary  if $\P(0)=R$ (resp. $\C(0)=R$), the unitary operation will be denoted by $\1$. We denote unitary operads by $\Pu$ (resp. unitary cooperads by $\Cu$) to indicate that they are unitary. Every unitary operad (resp. cooperad) has a non-unitary suboperad (resp. non-unitary  quotient cooperad) which is  denoted by $\P$ (resp. $\C$) and is defined as $\P(0):=0$ and $\P(r):=\Pu(r)$ for $r\geq 1$ (resp. $\C(0):=0$ and $\C(r):=\Cu(r)$ for $r\geq 1$, the coproduct is called the reduced coproduct). A unitary operad $\P$ (resp. unitary cooperad $\C$) is called reduced if $\P(1)=R$ is the operadic unit (resp. $\C(1)=R$ is the cooperadic counit). A non-unitary operad $\P$ (resp. cooperad $\C$) is called reduced if $\P(0)=\P(1)=0$ (resp. $\C(0)=\C(1)=0$). Note that every reduced operad (resp. cooperad) is canonically augmented (resp. coaugmented). The non-unitary part of every reduced unitary operad has a natural weight grading which coincides with the arity grading.

We further assume for simplicity that all operads and cooperads we consider are of finite $R$-type, i.e. they are finitely generated $R$-modules in each degree in each arity. All operads and cooperads are assumed to be projective as $R$-modules in each arity.  Finally, we also assume that the non-unitary part of all cooperads is conilpotent. With these assumptions we can freely dualize between operads and cooperads. Note that the dual of a unitary operad is usually not conilpotent.



\subsection{(Co)Algebras over (co)operads}

In this section, we recall the details about (co)algebras over (co)operads and their up to homotopy  versions.

\subsubsection{Invariants vs coinvariants}

Since we are working over general rings instead of a field of characteristic $0$, there are a few subtleties in the definition of algebras over symmetric operads. The main issue is that when working over an arbitrary ring, there are two ways to consider the symmetric group actions, we can either take invariants or coinvariants with respect to the symmetric group actions. As is explained in \cite{Fresse00}, over a field  characteristic $0$ these choices are isomorphic, but over an arbitrary ring these choices can be very different. In this section, we recall the necessary definitions and fix our conventions about algebras and coalgebras over operads and cooperads, for more details see \cite{Fresse00}.

Given two symmetric sequences $M$ and $N$, the tensor product of $M$ and $N$ is defined as
\[
\left( M \otimes N \right) (r)= \bigoplus_{i+j=r} \Ind^{\S_r}_{\S_i \times \S_j} M(i)\otimes N(j),
\]
where $\Ind$ denotes the induced representation. Besides the tensor product, there are also two version of the composition product. The first version, which is just called the composition product, is defined as
\[
M \circ N = \bigoplus_{r \geq 0} \left( M(r) \otimes N^{\otimes r} \right)_{\S_r}. 
\]
The composition product with divided symmetries is defined as
\[
M \tilde{\circ} N = \bigoplus_{r \geq 0} \left( M(r) \otimes N^{\otimes r} \right)^{\S_r}, 
\]
so instead of using coinvariants we are using invariants with respect to the symmetric groups $\S_r$.

Using these two versions of the composition product we can define algebras  (resp. coalgebras) over operads (resp. cooperads), again we have two different versions here, one using invariants and one using coinvariants. 

Let $V$ be a chain complex and $\mathcal{P}$ an operad. Then we call $V$ a $\P$-algebra if there exists a map 
\[
\gamma_V:\P\circ V \rightarrow V
\]
satisfying the usual axioms. The chain complex $V$ is called a $\P$-algebra with divided symmetries if there exists a map 
\[
\tilde{\gamma_V}:\P \tilde{\circ} V \rightarrow V,
\]
again satisfying the usual axioms.

Let $\C$ be a (non-unitary) reduced operad, the cofree conilpotent $\C$-coalgebra on a chain complex $V$, denoted $\C(V)$, is  defined as 
\[
\C(V):=\bigoplus_{r \geq 1} \left(\C(r) \otimes V^{\otimes r} \right)_{\S_r}.
\]
Note that the direct sum starts at $1$ because we assumed that $\C$ is reduced. The coproduct 
is induced by the cooperad structure of $\C$ (see \cite{LV} for more details).

Let $M$ and $N$ be two symmetric sequences, assume that $M(r)$ is free as an $R[\S_r]$-module in each arity $r\geq 0$, then we there is an isomorphism between the composition product and the composition product with divided powers. This isomorphism is induced by the norm map. If $X$ is an chain complex with an $\S_r$ action then there is a natural map, called the norm map, from the coinvariants to the invariants given by
\[
Tr:X_{\S_r} \rightarrow X^{\S_r}
\]
\[
Tr(x)=\sum_{\sigma \in \S_r} \sigma x,
\]
with $x \in X_{\S_r}$. The norm map is in general not an isomorphism, but if $X$ is free as an $R[\S_r]$-module then the norm map is an isomorphism. When $\Q \subseteq R$ there is a variation of the norm map given by $Tr(x)=\sum_{\sigma \in\S_r} \frac{\sigma x}{r!}$, which is also an isomorphism. Since we assumed that $M$ is free as an $R[\S_r]$-module, the norm map induces an isomorphism
\begin{equation}\label{eq:Normmap}
    Tr_{M,N}:M\circ N \rightarrow M\tilde{\circ} N.
\end{equation}
So in case that $\P$ (resp. $\C)$ is an operad (resp. cooperad) which  free as an $R[\S_r]$-module in each arity, there is no difference between algebras (resp. coalgebras) and algebras with divided symmetries (resp. coalgebras with divided symmetries).



\subsubsection{Homotopy algebras and the bar and cobar construction}

In this section, we discuss algebras up to homotopy and set up the situation in which things can be twisted.

Let $\P$ be a reduced operad (without unitary operation), if $\Q \nsubseteq R$, we further assume that the action of $\S_r$ is free in each arity $r$.  Let $\C$ be a reduced cooperad (also non-unitary) together with a Koszul twisting morphism $\tau:\C \rightarrow \P$. Then we define a $\P_\infty$-algebra as an algebra over $\Omega \C$, the operadic cobar construction of $\C$.

\begin{remark}
This definition varies a little bit from the definition of a $\P_\infty$-algebra from \cite{LV}, in their definition the operad $\P$ would be assumed to be quadratic and Koszul and $\C$ would be the quadratic dual. For preciseness, the cooperad $\C$ could have been included in the notation of a $\P_\infty$-algebra but in this paper it will always be clear which cooperad $\C$ we refer to.
\end{remark}

Alternatively, we can also define a $\P_\infty$-algebra on agraded $R$-module $V$ as a square-zero coderivation on the cofree conilpotent $\C$-coalgebra with divided symmetries $\C(V)$. This follows from the Rosetta Stone Theorem (Theorem 10.3.1 in \cite{LV}), because we assumed that all the symmetric group actions are free, this extends to arbitrary rings.

\begin{proposition}\label{prop:rosettastone}
Let $V$  be a chain complex  and $\tau:\C \rightarrow \P$ be a Koszul twisting morphism between a reduced cooperad $\C$ and a reduced operad $\P$, if $\Q \nsubseteq R$ further assume that $\C$ and $\P$ are free as $R[\S_r]$-modules in each arity. Then the notion of a $\P_\infty$-algebra, i.e. a morphism of operads $\gamma:\P_\infty=\Omega \C \rightarrow \End(V)$ is equivalent to a square-zero coderivation $Q:\C(V)\rightarrow \C(V)$ of degree $-1$, where $\C(V)$ is the cofree conilpotent $\C$-coalgebra.
\end{proposition}

Since $\C(V)$ is cofree, it further turns out that the coderivation $Q$ is determined by its image on cogenerators.

\begin{lemma}\label{lem:coderivation}
Under the assumptions of Proposition \ref{prop:rosettastone}, the coderivation $Q:\C(V)\rightarrow \C(V)$ is determined by $\tilde{Q}:\C(V)\rightarrow V$, the projection onto the cogenerators of $\C(V)$.
\end{lemma}

The proof is completely analogous to the proof of Proposition 6.3.8 of \cite{LV} and is therefore omitted. The converse of the lemma is not true, not every map $f:\C(V)\rightarrow V$ determines a square-zero coderivation. Every degree $-1$ map $f:\C(V)\rightarrow V$ does determine a coderivation but this might not square to zero. 

The map $\tilde{Q}:\C(V)\rightarrow V$ can be decomposed to define the set of multiplications on $V$. We denote the arity $r$-component by $\tilde{Q}_r:\C(r)\otimes V^{\otimes r} \rightarrow V$. Every element $\delta \in \C(r)$ defines a multiplication $Q_\delta:V^{\otimes r}\rightarrow V$ by first taking the inclusion  $\C(r)\otimes V^{\otimes r} \hookrightarrow \bigoplus_{r \geq 1} \left( \C(r)\otimes V^{\otimes r} \right)_{\S_r}$ and then the projection onto $V$.

\subsubsection{Filtrations and completions}

In the rest of this paper, we need to consider certain infinite sums, we therefore need to include certain filtrations. We mainly use the conventions from \cite{deKleijnWierstra} (except that in this paper we use a homological grading instead of the cohomological one in \cite{deKleijnWierstra}).

Let $V$  be a chain complex then we will from now on assume that $V$ is equipped with a descending filtration 
\[
V=F^1 V\supset F^2 V\supset F^3 V \supset ...,
\]
whichis respected by the differential and  satisfies $\cap_{k}F^kV=\{0\}$. We further assume that the filtration is complete, i.e. $V =\varprojlim F^i V$. 

The tensor product of two chain complexes is defined by first taking the tensor product of the filtrations and then completing with respect to this induced filtration (for more details see \cite{Fresse17} Section 7.3). We further need that if $W$ is a finitely generated chain complex and $V$ a filtered chain complex then $\hom(W,V)$ is also filtered with filtration
\[
F^n\hom(W,V):=\hom(W,F^n V).
\]
If the filtration on $V$ is complete then so is the filtration on $\hom(W,V)$. Similary, we can equip the tensor product of a filtered chain complex $V$ and an unfiltered finitely generated chain complex $W$ with a filtration given by 
\[
F^n(W \otimes V):=W \otimes F^n V.
\]
Again if the filtration on $V$ is complete then so is the filtration on $W \otimes V$.

Using this completed tensor product, the $r$-fold tensor product of a dg-$R$-module $V$ with itself also becomes a complete $R$-module. For a cooperad $\C$, we can extend this induced filtration to $\C(V)$.

An  $\Omega \C$-algebra is called complete or pro-nilpotent if the square-zero coderivation  $\tilde{Q}:\Cu(V)\rightarrow V$ respects the filtrations, i.e. if 
\[
\tilde{Q}_r\left(\C(r) \otimes F^{i_1}V \otimes ... \otimes F^{i_r} V\right) \subseteq F^{i_1+...+i_r}V.
\]

\begin{convention}
From now on we will always assume that the objects we are working with have a filtration and are complete with respect to this filtration.
\end{convention}

\subsubsection{Curved $\P_\infty$-algebras and $\infty$-morphisms}

If we assume that $\Cu$ is reduced unitary operad and that all our chain complexes are equipped with complete filtrations, then we can generalize $\P_\infty$-algebras to curved $\P_\infty$-algebras.  Curved algebras where first introduced by Positselski in \cite{Positselski}, where he showed the remarkable fact that there is a Koszul duality between unital algebras and curved coalgebras. 

\begin{definition}
    Let $\Cu$ be a unitary reduced cooperad with non-unitary part $\C$ and $V$  a filtered graded $R$-module. Since $\Cu$ is reduced, its non-unitay part $\C$ automatically has a weight grading. We denote by $\widehat{\Cu(V)}$ the completion of $\bigoplus_{r\geq 0}\left( \Cu(r) \otimes V^{\otimes r} \right)_{\S_r}$ with respect to this weight grading, where we put $\1\in \Cu(0)$ in weight $0$. We call this the completed conilpotent  cofree $\C$-coalgebra. 
\end{definition}

By using the cooperadic decomposition map $\Cu(V)$ becomes a $\Cu$-coalgebra.

\begin{definition}
Let $\Cu$ be a unitary cooperad and $\C$ its non-unitary part. A curved $\Omega\C$-algebra on a graded $R$-module $V$ is defined as a square-zero coderivation $Q:\Cu(V)\rightarrow \Cu(V)$. Similar to Lemma \ref{lem:coderivation}, every square-zero coderivation is determined by its image on the cogenerators $\tilde{Q}:\Cu(V)\rightarrow V$.

The curvature $\RR$ of a curved $\Omega\C$-algebra $(V,\tilde{Q}:\Cu(V)\rightarrow V)$ is defined as $\RR:=\tilde{Q}(\1)\in V$, where $\1$ is the unitary operation of $\Cu$. We call a curved $\Omega \C$-algebra flat if $\RR=\tilde{Q}(\1)=0$. 
\end{definition}

\begin{remark}
For simplicity, we have chosen to only use the definition of curved coalgebras as coderivations on the cofree conilpotent unitary $\Cu$-coalgebra.  There is also a theory of curved coalgebras using curved operads, it seems plausible that most of the theory would also work in that description but that is beyond the scope of this paper. For more details, see for example \cite{DCM}.
\end{remark}

For the definition of the exponential map we further need the notion of tangent vector of an element 

\begin{definition}\label{def:tangentvector}
    Let $\tilde{Q}:\Cu(V)\rightarrow V$ be a curved complete $\Omega \C$-algebra, then the cooperadic counit map $\epsilon:\Cu\rightarrow I$ induces a map $T:\Cu(V)\rightarrow I(V) \cong V$, where $I(V)$ is the cofree coalgebra generated by the cooperad $I$. We call this the tangent space of $\Cu(V)$. If $v\in \Cu(V)$ then we call $T(v)$ the tangent vector of $v$.
\end{definition}

\begin{lemma}
    Let $V$ be a flat curved $\Omega \C$-algebra, then $V$ is an algebra over the operad $\Omega \C$.
\end{lemma}

\begin{proof}
Since $V$ is a flat curved $\Omega \C$-coalgebra it is defined by a coderivation $\tilde{Q}:\Cu(V)\rightarrow V$ with $\tilde{Q}(\1)=0$. We can therefore restrict the coderivation $\tilde{Q}$ to a coderivation $\tilde{Q}':\C(V)\rightarrow V$, this coderivation also squares to zero because $\tilde{Q}$ squared to zero and $R=0$. It then follows from the Rosetta Stone (Proposition \ref{prop:rosettastone}) that this is an $\Omega \C$-algebra.
\end{proof}

\begin{definition}
A curved $\Omega \C$-algebra is called nilpotent if there exists an $N \in \N$ such that the arity $r$-part of the coderivation $\tilde{Q}:\Cu(V)\rightarrow V$ is zero for all elements of arity $r >N$.
\end{definition}

Curved $\Omega \C$-algebras have two types of morphisms, strict morphisms and $\infty$-morphisms. Let $A$ and $B$ be two curved $\Omega \C$-algebras together with coderivations $\tilde{Q}^A:\Cu(A)\rightarrow A$ and $\tilde{Q}^B:\Cu(B)\rightarrow B$, then a strict morphism is a linear map $f:A \rightarrow B$ that strictly commutes with all structure maps and filtrations, i.e. if $a_1,...,a_r \in A$ and $\delta\in\Cu(r)$ then $f\tilde{Q}^A_r\left(\delta \otimes a_1 \otimes ... \otimes a_r\right)=Q^B_r\left(\delta\otimes f(a_1)\otimes... \otimes f(a_r)\right) $. The more interesting notion of a morphism is called an $\infty$-morphism and is defined as follows. An $\infty$-morphism $\Phi:A \rightsquigarrow B$ is a $\Cu$-coalgebra map 
\[
\Phi:\Cu(A)\rightarrow \Cu(B),
\]
such that $\Phi \circ Q^A=Q^b \circ \Phi$. Since $\Cu(B)$ is cofree, this is equivalent to a sequence of maps 
\[
\phi_r:\Cu(r)\otimes A^r \rightarrow B,
\]
satisfying a certain equation coming from the cooperadic decompositon map. We call the map $\phi_r$ the arity $r$ component of $\Phi$. Note that an $\infty$-morphism for which $\phi_r=0$ for all $r\neq 1$ is the same as a strict morphism.

\begin{remark}
    There is a second notion of an $\infty$-morphism, which is given by replacing $\Cu(A)$ and $\Cu(B)$ by their completions $\widehat{\Cu(A)}$ and $\widehat{\Cu(B)}$. These notions can be very different, for example in the non completed case the only grouplike element in $\Cu(A)$ (resp. $\Cu(B)$) is $\1_A$ (resp. $\1_B$) so $\Phi(\1_A)=\1_B$. In the completed case, this is no longer the case and the homotopy theory of these algebras behaves very differently (see for example \cite{GuanLazarev} in the case of modules). Since the non-completed notion of $\infty$-morphism coincides with the classical notion of $\infty$-morphism we have chosen to only focus on this notion.
\end{remark}

\section{The (generalized) shuffle product}

In this section, we recall a generalization of the shuffle product, but before we can do that we first recall the notion of a Hopf (co)operad. For more details see for example \cite{Moerdijk} or \cite{Loday}.

\subsection{Hopf (co)operads and the tensor product of (co)algebras}

Let $\P$ be an operad and $A$ and $B$ two $\P$-algebras, in general the tensor product $A \otimes B$ does  not have the structure of a $\P$-algebra but only the structure of a $\P\otimes \P$-algebra, where $\P \otimes \P$ denotes the aritywise tensor product of the operad $\P$ with itself. There is a special class of operads called Hopf operads for which the tensor product has a natural $\P$-algebra structure. 

A Hopf operad $\P$ is an operad in the category of coassociative coalgebras, i.e. in each arity $r$ we have a coproduct $\Delta_r:\P(r)\rightarrow \P(r) \otimes \P(r)$ such that the operadic composition maps commute with the coproducts. A Hopf operad is called counital if the coproducts $\Delta^r$ are all counital.

Suppose that $A$ and $B$ are algebras over a Hopf operad $\P$ then we can equip the tensor product $A\otimes B$ with a $\P$-algebra structure where the multiplication map
\[
\P(r)\otimes \left(A \otimes B\right)^{\otimes r} \rightarrow A \otimes B
\]
is given by the following map
\[
\P(r)\otimes \left( A \otimes B \right)^{ \otimes r} \xrightarrow{\Delta_r \otimes \id_{\left(A\otimes B \right)^{\otimes r}}} \P(r)\otimes \P(r)  \left(A \otimes B\right)^{\otimes r} \xrightarrow{\tau} \P(r)\otimes A^{\otimes  r} \otimes \P(r)\otimes B^{\otimes r} \xrightarrow{\gamma_A\otimes\gamma_B} A\otimes B,
\]
where $\tau$ is the map shuffling the tensor factors and $\gamma_A$ and $\gamma_B$ are the structure maps of $A$ and $B$.

Similarly, we can define Hopf cooperads, these are cooperads in the category of associative algebras. More precisely, a Hopf cooperad $\C$ is a cooperad together with maps $\mu_r:\C(r)\otimes \C(r)\rightarrow \C(r)$, such that each $\mu_r$ forms an associative product on $\C(r)$ and the maps $\mu_r$ commute with the cooperadic decomposition. We remark that technically this should be called a co-Hopf cooperad (as it is called in \cite{GetzlerJones}), but by a small abuse of terminology we just call it a Hopf cooperad. A Hopf cooperad is called unital if all the products $\mu_r$ are unital, the units are denote by $\eta_r\in \C(r)$. Since the cooperadic decomposition maps respect the units, every unitary unital Hopf cooperad $\Cu$ comes equipped with a map $\eta:u\Cocom\rightarrow \Cu$, where $u\Cocom$ is the unitary cocommutative cooperad. The map is given by sending the arity $r$ operation of $u\Cocom$ to $\eta_r$.

Similar to algebras, it turns out that the tensor product of two coalgebras with divided symmetries $C$ and $D$ over a Hopf cooperad $\C$ is naturally a $\C$-coalgebra with divided symmetries. The coproduct
\[
\delta_r^{C\otimes D}:C\otimes D \rightarrow \left(\C(r)\otimes \left( C \otimes D \right)^{\otimes r}\right)_{\S_r}
\]
is defined in a dual way as in the algebra case and is explicitly given by the following composition
\[
C\otimes D \xrightarrow{\delta^C_r \otimes \delta^D_r} \left(\C(r)\otimes C^{\otimes r} \otimes \C(r) \otimes D^{\otimes r}\right)_{\S_r} \xrightarrow{\tau} \left(\C(r)\otimes \C(r) \otimes \left(C\otimes D\right)^{\otimes r}\right)_{\S_r} 
\]
\[
\xrightarrow{\mu_r\otimes \id_{\left(C \otimes D\right)^{\otimes r}}}\left( \C(r)\otimes   \left(C \otimes D\right)^{\otimes r}\right)_{\S_r},
\]
where $\delta^C_r$ and $\delta^D_r$ are the the coproduct maps of $C$ and $D$ and $\tau$ is the map that permutes the tensor factors.

\subsection{The generalized shuffle product}

It is a well known fact that the counital cofree conilpotent cocommutative coalgebra on a chain complex $V$ is not just a cocommutative coalgebra but also carries an additional product called the shuffle product. With this shuffle product, the cofree conilpotent cocommutative  coalgebra becomes a Hopf algebra. In this section, we recall a generalization of the shuffle product which was originally due to Moerdijk (see Example 2.3 of \cite{Moerdijk}). Moerdijk showed that the free algebra over a Hopf operad carries a coassociative coproduct. For the constructions in this paper, we need the dual of Moerdijk's construction which we describe here. For more details see also Section 3.2.1 of \cite{Loday}.

Let $\Cu$ be a reduced unitary Hopf cooperad and $V$ a chain complex. Since $\Cu$ is reduced it is canonically augmented. Then $\Cu(V)$, the cofree conilpotent coassociative coalgebra on $V$, has an associative product
\[
\star:\Cu(V)\otimes \Cu(V) \rightarrow \Cu(V)
\]
which is called the generalized shuffle product and is defined as follows. Because $\Cu$ is a Hopf cooperad, the tensor product $\Cu(V)\otimes \Cu(V)$ becomes a $\Cu$-coalgebra as well.  Since $\Cu(V)$ is cofree, we only need to specify what the image is on the cogenerators of $\Cu(V)$. The coaugmentation $\eta:I\rightarrow \Cu$ of  $\Cu$ induces a map $V\cong I(V) \rightarrow \Cu(V)$, by abuse of notation we will denote the image of this map by $V$. The restriction to cogenerators $\tilde{\star}:\Cu(V)\otimes \Cu(V) \rightarrow V$ of the generalized shuffle product is defined by 
\[
\1\star v =v \star \1 =v,
\]
for $v \in V$ and zero otherwise. By cofreeness this extends to a product $\star:\Cu(V)\otimes \Cu(V)\rightarrow \Cu(V)$. We call this product the generalized shuffle product.

The generalized shuffle product has the following compatibility with the coproducts of $\Cu(V)$.

\begin{lemma}\label{lem:shuffleproduct}
   Let $\Cu$ be a unitary Hopf cooperad and $V$ a graded $R$-module. The generalized shuffle $\star:\Cu(V)\otimes \Cu(V) \rightarrow \Cu(V)$ product is a morphism of $\Cu$-coalgebras, in particular the following diagram commutes
  \begin{center} 
   \begin{tikzcd}
       \Cu(V) \otimes \Cu(V) \arrow[rr,"\Delta_{\Cu(V)\otimes\Cu(V)}"] \arrow[d,"\star"]  & &\Cu(\Cu(V)\otimes \Cu(V)) \arrow[d,"\Cu(\star)"]\\
       \Cu(V)\arrow[rr,"\Delta_{\Cu(V)}"]  & &  \Cu(\Cu(V)),
   \end{tikzcd}
\end{center}
where $\Delta_{\Cu(V)\otimes\Cu(V)}$ and $\Delta_{\Cu(V)}$ are the coalgebra maps.
\end{lemma}

The proof of the lemma is omitted since it is a straightforward consequence of the definition of the generalized shuffle product and the compatibility of the Hopf product with the cooperad structure. By the same arguments as in Appendix A of \cite{deKleijnWierstra}, we can extend the generalized shuffle product to the completed $\Cu$-coalgebras.

\begin{lemma}
    Let $\Cu$ be a reduced unitary Hopf cooperad and $V$ a graded $R$-module. The generalized shuffle product extends to the completed cofree $\Cu$-coalgebras
    \[
    \star:\widehat{\Cu(V)}\otimes \widehat{\Cu(V)}\rightarrow \widehat{\Cu(V)}. 
    \]
\end{lemma}

Both $\Cu(V)$ and $\widehat{\Cu(V)}$ therefore become $\Cu$-$\Ass$-bialgebras.

\section{The twisting procedure for algebras over Hopf operads}\label{sec:Lietheory}

In the previous section, we saw that $\widehat{\Cu(V)}$, the completed cofree conilpotent $\Cu$-coalgebra on a graded $R$-module $V$ over a Hopf cooperad $\Cu$, becomes a $\Cu$-$\Ass$-bialgebra with the generalized shuffle product. In this section, we explain how this bialgebra structure can be interpreted as the analog of a Lie group. In this interpretation, the "group structure" is given by the  generalized shuffle product  and the manifold structure is given by the $\Cu(V)$-coalgebra structure and its differential $Q:\Cu(V)\rightarrow \Cu(V)$. We assume that the reader is familiar with the basics of classical Lie theory, otherwise we refer the reader to for example \cite{Kirillov} for an introduction to classical Lie theory.

In this analogy, the tangent space of the "Lie group" $\Cu(V)$ is given by $V$.  It turns out that many concepts from the classical Lie theory of manifolds also extend to our setting. For example, the exponential map and the adjoint representation both have analogs in this setting and together they give rise to the twisting procedure.   

Most of the ideas in this section can be seen as a generalization of Dolgushev's approach to the twisting procedure of $L_\infty$-algebras from \cite{Dolgushev}. This approach to the twist was later extended to $A_\infty$-algebras by the authors. 

\subsection{The exponential map}

In this section, we define the exponential map and show that it has properties similar to the exponential map in the theory of classical Lie groups.  Recall that for  a Lie group $G$ with Lie algebra $\mathfrak{g}$, the exponential map is a function $\exp:\mathfrak{g}\rightarrow G$ that maps $\mathfrak{g}$  to $G$ and captures its local structure. For an element $v \in \mathfrak{g}$, it is explicitly defined by first  considering  the unique one parameter subgroup $\gamma_v:\R \rightarrow G$ whose tangent vector at the identity is equal to $v$. The exponential map of $v$ is then defined as $\exp(v)=\gamma_v(1)$, where $1\in \R$ is the multiplicative unit of $\R$.

We want to generalize this construction to $\Cu$-$\Ass$-bialgebras. For this we first need to find the analog of the real line, which is given by $R[R]$, the group ring of $R$ as an additive group with coefficients in $R$.  This needs to be turned into a $\Cu$-$\Ass$-bialgebra where the $\Cu$-bialgebra structure has divided symmetries. The group ring $R[R]$ has a natural cocommutative  coproduct given by defining the elements of the additive group $R$ as grouplike, this needs to be extended to a $\Cu$-coalgebra structure. To do this we use that every unitary unital Hopf cooperad $\Cu$ comes with a natural map $\eta:u\Cocom \rightarrow \Cu$.



\subsubsection{The analog of the real line}

As mentioned earlier, the analog of the real line $\R$ in our setting is played by the group ring $R[R]$. In the setting of classical Lie groups, the real line $\R$ has a group structure given by addition and a manifold structure given by the usual manifold structure on $\R$. In our setting the group ring $R[R]$ has a "group" structure (strictly speaking monoid structure) but not yet a $\Cu$-coalgebra structure, which would be the analog of the manifold structure on $\R$.

To avoid confusion between $R$ as a coefficient ring and $R$ as an additive group, we will from now on denote $R$ as an additive group by $G$. The elements of $G$ are denoted by $g_\lambda$ with $\lambda \in G$, since $G$ and $R$ are the same as additive groups we will occasionally also use $R$ as an indexing set for the elements $g_\lambda$. The additive unit of $R$ is denoted by $e$ and the multiplicative unit by $g_\1$.  With this notation the group ring has a basis given by the elements $g_\lambda$ and the ring structure of the group ring is defined on these basis elements by $g_\lambda \cdot g_\mu=g_{\lambda+\mu}$, with $\lambda,\mu \in G$. The group ring $R[G]$ further becomes a cocommutative coalgebra with the classical coproduct  
\[
\Delta:R[G] \rightarrow R[G] \otimes R[G]
\]
defined on the basis $\{g_\lambda\}$ by 
\[
\Delta(g_\lambda)=g_\lambda\otimes g_\lambda
\]
for $g_\lambda \in G$. 

The coproduct of $R[G]$ naturally lands in the invariants and not in the coinvariants, so this coalgebra does not have divided symmetries. Since we want to look at morphisms from $R[G]$ to $\Cu(V)$ we need to give $R[G]$ a $\Cu$-coalgebra structure with divided symmetries. 

First, we turn $R[G]$ into a $\Cu$-coalgebra (without divided symmetries) by using the morphism $\eta:u\Cocom \rightarrow \Cu$, so we have a map 
\begin{equation}\label{eq:groupalgebra}
R[G]\rightarrow \widehat{\bigoplus_{r \geq 0} \left( \Cu(r) \otimes R[G]^{\otimes r}\right)^{\S_r} }.
\end{equation}
 Since we assumed that $\Cu$ is free as an $R[\S_r]$-module in each arity, the invariants and coinvariants are isomorphic. We can therefore apply this isomorphism to Equation \ref{eq:groupalgebra} to get a map  
\[
R[G]\rightarrow \widehat{\bigoplus_{r \geq 0} \left(\Cu(r) \otimes R[G]^{\otimes r}\right)_{\S_r} },
\] 
 which turns the group algebra into a $\Cu$-coalgebra with divided symmetries.

\begin{remark}
When $\Q \subseteq R$,  the norm map is always invertible and in this case our constructions work for all unitary reduced unital Hopf cooperads. 
\end{remark}
 

\subsubsection{The exponential map}

The exponential map is defined as follows. We want to define a map which generalizes the exponential map of classical Lie algebras. In particular, we want to associate to each degree $0$ element $v\in V_0$ the unique "one parameter subgroup" of $\widehat{\Cu(V)}$ with tangent vector $v$. With this we mean a $\Cu$-$\Ass$-bialgebra map $\gamma_v:R[G] \to \widehat{\Cu(V)}$  from $R[G]$, the analog of the real line, to $\widehat{\Cu(V)}$ with $\gamma(e)=\1$ and $T(\gamma_v(g_\1))=v$, i.e. the tangent vector at the multiplicative unit $g_\1\in G$ is given by $v$ (see Definition \ref{def:tangentvector}).   This does not necessarily make the map $\gamma_v$ unique, but it becomes unique once we further require that $T(\gamma_v(\lambda g_\1))=\lambda T(\gamma_v(g_\1))$, for $\lambda \in R$. So in particular if the "speed" of the exponential becomes $\lambda$ times bigger, the tangent vector needs to become $\lambda$ times bigger as well.

\begin{proposition}\label{prop:gammav}
Let $v \in V_0$, there is a unique map $\gamma_v:R[G]\rightarrow \widehat{\Cu(V)}$ of $\Cu$-$\Ass$-bialgebras determined by the following properties:
\begin{enumerate}
    \item The tangent vector of $\gamma_v:R[G]\rightarrow \widehat{\Cu(V)}$ at the unit is given by $T(\gamma_v(e))=v$ and $\lambda\in R$ we have that $T(\gamma_v(\lambda \cdot e))= \lambda \cdot v$.
    \item The map $\gamma_v:R[G]\rightarrow \widehat{\Cu(V)}$ is a morphism of bialgebras, i.e. we have $\gamma_v(\alpha \cdot \beta)=\gamma_v(\alpha)\star \gamma_v(\beta)$, with $\alpha,\beta \in R[G]$.
\end{enumerate}
\end{proposition}

\begin{proof}
Since $\widehat{\Cu(V)}$ is completed cofree every map of $\Cu$-coalgebras is determined by its image on the cogenerators. As a morphism of $\Cu$-coalgebras, the map $\gamma_v$ is therefore uniquely defined by the $R$-linear map 
\[
\widetilde{\gamma_v}:R[G]\rightarrow V
\]
which is defined on basis elements $g_\lambda\in R[G]$ (with $\lambda \in R$) by
\[
\widetilde{\gamma_v}(g_\lambda)=\lambda \cdot v.
\]
This determines $\gamma_v$ uniquely as a $\Cu$-coalgebra map, what is left to check is whether this map is also a bialgebra map. So we need to check that it commutes with the associative products of $R[G]$ and $\widehat{\Cu(V)}$.

To show that the exponential map is a morphism of bialgebras, we need to check that 
\[
\gamma_v(g_\alpha \cdot g_\beta)=\gamma_v(g_\alpha) \star \gamma_v(g_\beta).
\]
This is equivalent to showing that the following diagram of $\Cu$-coalgebra maps commutes:
\[
\begin{tikzcd}
R[G] \otimes R[G] \arrow[d,"\mu_{R[G]}"]  \arrow[r,"\gamma_v\otimes\gamma_v"] &\Cu(V) \otimes \Cu(V) \arrow[d,"\star"]  \\
R[G] \arrow[r,"\gamma_v"]   & \Cu(R[G]),
\end{tikzcd}
\]
where $\star$ is the generalized shuffle product and $\mu_{R[G]}$ is the multiplication of the group ring $R[G]$. Since $\Cu(R[G])$ is cofree in the completed sense, every map is determined by its image on the cogenerators. We therefore only need to check that the maps $\gamma_v \circ \mu_{R_{[G]}}$ and $\star \circ\left( \gamma_v \otimes \gamma_v \right)$ have the same tangent vector, which is equal to its image on the cogenerators. In other words, we have to show that for $\alpha \otimes \beta \in R[G] \otimes R[G]$
\begin{equation}\label{eq:tangentvectors}
T\left(\gamma_v \circ \mu_{R_{[G]}}\right)=T\left(\star \circ\left( \gamma_v \otimes \gamma_v \right)\right).
\end{equation}

If we do this explicit computation, we see that the left hand side of Equation \ref{eq:tangentvectors} is equal to
\[
T(\gamma_v \circ \mu_{R_{[G]}}(\alpha \otimes \beta))=\alpha +\beta.
\]  
The right hand side of Equation \ref{eq:tangentvectors} is computed as follows. First of all, the elements $\gamma_v(\alpha)$ (resp. $\gamma_v(\beta)$) are of the form $1+\alpha + \mbox{higher order terms}$ (resp. $1+\beta + \mbox{higher order terms}$), where the higher order terms are given by coproducts of arity $2$ and greater. The generalized shuffle product is then given by $\gamma_v(\alpha)\star \gamma_v(\beta)=1+\alpha + \beta + \mbox{higher order terms}$, so after projecting on cogenerators it is given by $\alpha+\beta$. The maps therefore have the same projection onto the cogenerators and are therefore equal. The map $\gamma_v$ is therefore a map of bialgebras.    
\end{proof}

From these properties, it follows that $\gamma_v$ is determined by the image of $g_{\1}$.

\begin{lemma}
From the properties of Properties \ref{prop:gammav}, it also follows that the $\Cu$-$\Ass$-bialgebra map $\gamma_v:R[G]\rightarrow \widehat{\Cu(V)}$ is equivalent to a $\Cu$-coalgebra map $\gamma_v':R \rightarrow \widehat{\Cu(V)}$.    
\end{lemma}

\begin{proof}
    It is clear that a $\Cu$-$\Ass$-bialgebra map $\gamma_v:R[G]\rightarrow \widehat{\Cu(V)}$ is determines a $\Cu$-coalgebra map $\gamma_v':R \rightarrow \widehat{\Cu(V)}$. This map is given by $\gamma_v'(\lambda):=\gamma_v(\lambda g_{\1})$, with $\lambda \in R$.

    To go in the other direction, we assume that we have a map $\gamma_v':R \rightarrow \widehat{\Cu(V)} $ and extend it to a map  $\gamma_v:R[G]\rightarrow \widehat{\Cu(V)}$. Since $\widehat{\Cu(V)}$ is cofree, we only need to define its projection on cogenerators, the image of the generators $g_\lambda\R[G]$ by the following formula
    \[
    T(\gamma_v(g_\lambda)):=T(\gamma_v'(\lambda)).
    \]
    It follows from the properties of Propositions \ref{prop:gammav} that this can indeed be extended to a bialgebra map. 
\end{proof}

Now that we have the definition of the analog of a  "one-parameter subgroup" we can define the exponential map in exactly the same way as it is done in classical Lie theory.

\begin{definition}
    The exponential map is defined by
    \[
    \exp:V_0\rightarrow \widehat{\Cu(V)}
    \]
    \[
    exp(v):=\gamma_v(g_\1).
    \]
\end{definition}

Similar to the classical exponential map, our exponential map has the following properties. 

\begin{proposition}\label{prop:propertiesofexp}
Let $v \in V_0$, $\kappa,\lambda \in R$ and $x\in \Cu(V)$. The exponential map has the following properties:
\begin{enumerate}
    \item $\exp(\kappa v) \star \exp(\lambda v)= exp((\kappa+\lambda )v)$.
    \item $\exp(-v)\star \exp(v)=\1=\exp(v)\star \exp(-v)$ or alternatively $\exp(-v)$ is a multiplicative inverse to $\exp(v)$.
    \item The map $\exp(v)\star -:\widehat{\Cu(V)}\rightarrow \widehat{\Cu(V)}$ is a morphism of $\Cu$-coalgebras.
\end{enumerate}
\end{proposition}

\begin{proof}
Part $1$ and part $2$ follow immediately from   Proposition \ref{prop:gammav} and their proof is omitted. We prove Part $3$ as follows. We need to show that the map $\exp(v)\star -:\widehat{\Cu(V)}\rightarrow \widehat{\Cu(V)}$ commutes with the coproduct, so if $x \in \widehat{\Cu(V)}$ then we need to show that  
\begin{equation}\label{eq:expisaCmorphism}
\Delta_{\widehat{\Cu(V)}} (\exp(v)\star x) =\Delta(\star)(\Delta_{\widehat{\Cu(v)}\otimes \widehat{\Cu(V)}}( \exp(v) \otimes x)),
\end{equation}
with $x \in \widehat{\Cu(V)}$ and $\Delta(\star)$ is the coproduct of the generalized shuffle product $\star$. To show that Equation \ref{eq:expisaCmorphism} holds, we need two ingredients. The first one is Lemma \ref{lem:shuffleproduct}   and the second one is that $\exp(v)$ satisfies a grouplike property. Since the element $g_{\1}$ is grouplike in $R[G]$, its coproduct is given by $\Delta(g_{\1})=\sum_{r \geq 0}\eta_r \otimes g_{\1}^{\otimes r}$ ,where $\eta_r$ is the image of $u\Cocom(r)$ in $\Cu(r)$. Since $\gamma_v$ is a morphism of $\Cu$-bialgebras, the coproduct of $\Delta(\exp(v))=\Delta(\gamma_v(g_{\1}))=\sum_{r \geq 0} \eta_r \otimes \exp(v)^{\otimes r}$. If we now use that $\eta_r$ the unit is for the product $\mu_r:\Cu(r) \otimes \Cu(r) \rightarrow \Cu(r)$ and Lemma \ref{lem:shuffleproduct}, we see that Equation \ref{eq:expisaCmorphism} holds.
\end{proof}

The following corollary  follows almost immediately from Proposition \ref{prop:propertiesofexp}. Note that in second part of the following proposition we do not use the completion of the coalgebra $\Cu(V)$.

\begin{corollary}
The element $\exp(v)$ is invertible in $\widehat{\Cu(V)}$ with the generalized shuffle product, its inverse is given by $\exp(-v)$. The morphism
\[
\exp(v)\star -: \Cu(V) \rightarrow \Cu(V)
\]
is therefore an isomorphism with inverse $\exp(-v)\star -$.
\end{corollary}

\begin{proof}
    The fact that $\exp(v)\star -$ is invertible follows immediately from Proposition \ref{prop:propertiesofexp}. The fact that the map $\exp(v)$ preserves the subspace $\Cu(V) \subset \widehat{\Cu(V)}$ follows from the same arguments as in Appendix A of \cite{deKleijnWierstra}.
\end{proof}



\subsection{The twisting procedure and the Maurer-Cartan equation}

Using the exponential from the previous section, we can define the twisting procedure for $\Omega \C$-algebras. This is similar to the adjoint representation of a Lie group on its Lie algebra. 

\begin{definition}
    Let $(\Cu(V),Q)$ be a curved $\Omega \C$-algebra and $v\in V_0$. Then we define $Q^v:\Cu(V)\rightarrow \Cu(V)$, the differential of $(\Cu(V),Q)$ twisted by $v$ as
    \[
    Q^v:\Cu(V)\rightarrow \Cu(V)
    \]
    \[
    Q^v(x):=\exp(-v)\star Q\left(\exp(v)\star x\right),
    \]
    with $x \in \Cu(V)$.
\end{definition}

\begin{theorem}\label{thrm:twistingprocedure}
    The twisted differential $Q^v$ is a coderivation and squares to zero. The twist of a curved $\Omega \C$-algebra is therefore again a curved $\Omega \C$-algebra. 
\end{theorem}

\begin{proof}
    We need to show that the twisted differential $\Q^v$ is again a coderivation and squares to zero. First note that by Proposition \ref{prop:propertiesofexp}, multiplication by $\exp(v)$ is a $\Cu$-coalgebra isomorphism with inverse $\exp(-v)$. Since the conjugation of a coderivation by an isomorphism is again a coderivation, the map $Q^v$ is again a coderivation. 

    The map $Q^v$ squares to zero since 
    \begin{align*}
    (Q^v)^2(x)&=\exp(-v)\star Q(\exp(v)\star \exp(-v)\star Q(\exp(v) \star x))\\
    &= \exp(-v)\star Q (\1 \star Q(\exp(v) \star x)) \\
    &= \exp(-v) \star Q^2(\exp(v) \star x) \\
    &=0
    \end{align*}
    where we used that multiplication by $\1$ is the identity. So the twisted differential squares again to zero which proves the theorem.
\end{proof}



Using this notion of the twist, we can define Maurer-Cartan elements as those elements that produce a flat $\Omega \C$-algebra after twisting.

\begin{definition}\label{def:MCequation}
Let $(\Cu(V),Q)$ be a curved $\Omega \C$-algebra and let $v\in V_0$, then the Maurer-Cartan equation is defined as
\[
\tilde{Q}(exp(v))=0.
\]
An element is called a Maurer-Cartan element if it satisfies the Maurer-Cartan equation, the set of Maurer-Cartan elements is denoted by $\MC(V)\subseteq V$.
\end{definition}

The Maurer-Cartan equation can be made more explicit by using the explicit morphism $\eta:u\Cocom \rightarrow \Cu$. If $\eta_r'$ denotes the basis element for $u\Cocom(r)$ and $\eta_r=\varphi(\eta_r')$ then the Maurer-Cartan equation can be rewritten as 
\[
\sum_{r \geq 0} \tilde{Q}_r\left(\eta_r\otimes v^{\otimes r}\right) =0. 
\]
Using the same arguments as in Section 4 of \cite{deKleijnWierstra}, we get the following propostion which shows that the Maurer-Cartan equation is indeed a flatness equation.

\begin{proposition}
    Let $(\Cu(V),Q)$ be a curved $\Omega \C$-algebra, then $(\Cu(V),Q^v)$, the twist of $Q$ by an element $v\in V_0$, is flat if and only if $v$ is a Maurer-Cartan element.
\end{proposition}

\section{The Maurer-Cartan simplicial set}\label{sec:MCsimplicialsets}

In the theory of $L_\infty$- and $A_\infty$-algebras, the set of Maurer-Cartan elements can be extended to a simplicial set which encodes all the relevant gauges. These Maurer-Cartan simplicial sets have many applications in deformation theory, rational homotopy theory and related fields. In this section, we  show that these Maurer-Cartan simplicial sets can be constructed in much greater generality. In particular, we show that one can construct a Maurer-Cartan simplicial set for every unitary reduced unital Hopf cooperad $\Cu$ with a map $\E_\infty \rightarrow \C$, where $\E_\infty$ denotes the cochains on the Barratt-Eccles operad and $\C$ is the non-unitary part of $\C$. 

\begin{convention}
    From now on we only work with flat $\Omega \C$-algebras.
\end{convention}

\subsection{The Maurer-Cartan simplicial set}

We construct the Maurer-Cartan simplicial set as follows. Suppose that $\Cu$ is a unitary reduced unital Hopf cooperad, then we have seen in Section \ref{sec:Lietheory} that there exists a morphism $u\Cocom \rightarrow \Cu$ and that we have an exponential map and a Maurer-Cartan equation. We further saw  that if $V$ is a curved $\Omega \C$-algebra, then the exponential map is defined via the "one parameter subgroup" $\gamma_v:R[G]\rightarrow \Cu(V)$, which in turn was determined by a morphism of $\Cu$-coalgebras $R\rightarrow \Cu(V)$. 

Since $R$ can be identified with $N_*(\Delta^0;R)$, the chains on $\Delta^0$ with coefficients in $R$, the set of Maurer-Cartan elements is equal to the set of morphisms of $N_*(\Delta^0;R)$, considered as a $\Cu$-coalgebra using the morphism $u\Cocom\rightarrow \Cu$. to $\Cu(V)$. 

To define the higher simplices of the Maurer-Cartan simplicial set, we use the higher standard simplices $\Delta^n$. It is well known that the collection of standard simplices forms a cosimplicial object in the category of simplicial sets (see \cite{GoerssJardine} for example). So if we apply $N_*(-,R)$, the normalized chains functor with coefficients in $R$, we get a cosimplicial object in the category of chain complexes. To shorten the notation a little, we will from now on omit the coefficients in the notation for the normalized chains and always implicitly assume they are the ring $R$. To define the Maurer-Cartan simplicial set, we need to turn this into a cosimplicial object in the category of $\C$-coalgebras. The Maurer-Cartan simplicial set is then morally defined as 
\begin{equation}\label{eq:MCsSet}
\MC_n(V):=\hom_{\C-\mbox{coalg}}(N_*(\Delta^n),\widehat{\Cu(V)}).
\end{equation}

To equip $N_*(\Delta^n)$ with a $\C$-coalgebra structure, we first use the fact that it is an $\E_\infty$-coalgebra. There are several choices possible for an $\E_\infty$-coalgebra structure on the normalized chains. There are for example the surjection operad and the Barratt-Eccles operad (see \cite{McClureSmith} and \cite{BergerFresse}). In the rest of this paper, we  use the Barratt-Eccles operad from \cite{BergerFresse} and instead of working with coalgebras over operads, we work with the dual of the Barratt-Eccles operad, which we denote by $\E_\infty$. Since the Barratt-Eccles operad is of finite type, coalgebras over an operad are equivalent to coalgebras over the dual cooperad. So to turn $N_*(\Delta^n)$ into a $\C$-coalgebra, we need to assume that we have a morphism of cooperads $\varphi:\E_\infty\rightarrow \C$. Since the Barratt-Eccles cooperad has a free symmetric  group action, we can identify invariants and coinvariants so there is no difference between coalgebras with divided symmetries and ordinary coalgebras.

Since we assumed that all $\Omega \C$-algebras are flat, we see that a flat $\Omega \C$-algebra structure on a chain complex $V$ is equivalent to a square-zero coderivation on $\C(V)$ (so the non counital version). In this case, the set of $n$-simplices of the Maurer-Cartan simplicial set from Equation \ref{eq:MCsSet} is given by the set of maps of $\C$-coalgebras
\begin{equation}\label{eq:MCsSet2}
    \MC_n(V):=\hom_{\C-\mbox{coalg}}(N_*(\Delta^n),\widehat{\C(V)}).
\end{equation}
This is because we can identify $\C(V)$ with the relative bar construction on $V$ (see \cite{LV}, Chapter 11). Recall that since $(\C(V),Q)$ is cofree, every map to it is determined by its image on cogenerators. The converse is however not true since not every map commutes with the differential $Q$. We can however form the convolution algebra $\hom_{R}\left(N_*(\Delta^n),V\right)$, which becomes an algebra over the convolution operad $\hom(\C,\Omega\C)$ (see \cite{BergerMoerdijk}). As is explained in Chapter 11 of \cite{LV}, there is a Maurer-Cartan equation in this convolution algebra whose solutions correspond to the maps that commute with the differential $Q$. To define this Maurer-Cartan equation, we first define the $\star_\tau$-operator.  

Let $\C$ be a cooperad and $\P$ an operad both with free symmetric group actions. Let $(C,\Delta_C)$ be a $\C$-coalgebra and $(V,\gamma_V)$ a $\P$-algebra, since the symmetric group actions are free there is no difference between divided symmetries and no divided symmetries. Further suppose that we have an operadic twisting morphism $\tau:\C\rightarrow \P$. In our case $\P$ is given by $\Omega \C$ and $\tau$ is given by the canonical twisting morphism $\iota:\C\rightarrow \Omega \C$. The $\star_\tau$-operator is the (non-linear) map of degree $-1$
\[
\star_{\tau}:\hom_{R}(C,V)\rightarrow \hom_{R}(C,V)
\]
defined as the composite
\[
\star_\tau(\psi):=C\xrightarrow{\Delta_C} \C\circ C\xrightarrow{\tau\circ \psi} \P \circ V \xrightarrow{\gamma_V} V.
\]
We further define, for $r \geq 2$, the operations 
\[
\mu_r:\hom_{R}(C,V)^{\otimes r} \rightarrow \hom_{R}(C,V)
\]
by 
\[
\mu_r(\psi_1,...,\psi_r):=\gamma_V\circ(\psi_1 \otimes...\otimes \psi_r) \circ\Delta^r_C,
\]
with $\psi_1,...,\psi_r \in \hom_{R}(C,V)$ and $\Delta^r_C:C\rightarrow \left(\C(r)\otimes C^{\otimes r}\right)_{\S_r}$ is the arity $r$ part of the coproduct of $C$. The operator $\star_\tau$ can be written as $\star_\tau(\psi)=\sum_{r \geq 2} \mu_r(\psi,...,\psi)$. 

A Maurer-Cartan element is a degree $0$ morphism  $\psi\in\hom_{R}(C,V)$ which satisfies the Maurer-Cartan equation, which is given by
\begin{equation}\label{eq:relativeMCequation}
\partial (\psi) +\star_{\tau}(\psi)=0,
\end{equation}
where $\partial$ is the differential of $\hom_{R}(C,V)$. 
The set of solutions to the Maurer-Cartan equation is denoted by $\MC\left(\hom_{R}(C,V)\right)$. From Chapter 11 of \cite{LV} we get the following proposition.

\begin{proposition}[\cite{LV}, Proposition 11.3.1]\label{prop:MCstarequation}
Under the previous assumptions we have the following bijection
\[
\MC\left(\hom_{R}(C,V)\right) \cong \hom_{\C-\mbox{coalg}}\left(C, (\C(V),Q) \right).
\]
\end{proposition}

In our specific situation, the twisting morphism is given by the canonical twisting morphism $\iota:\C \rightarrow \Omega \C$ (see \cite{LV}, Section 6.5). Because of Proposition \ref{prop:MCstarequation}, we see that have an equivalence between the set of $n$-simplices from Equation \ref{eq:MCsSet2} and the Maurer-Cartan elements in $\hom_{R}(N_*(\Delta^n),V)$.  Further notice that when $n=0$, the chains on $\Delta^0$ are isomorphic to $R$. So there is an isomorphism $\hom_{R}(N_*(\Delta^0),V)\cong V$. It is straightforward to see that under this isomorphism the Maurer-Cartan equation in $\hom_{R}(N_*(\Delta^0),V)$ given by $\partial(\psi)+\star_\iota(\psi)$ is equivalent to the Maurer-Cartan equation from Definition \ref{def:MCequation}. An equivalent formulation of the Maurer-Cartan simplicial set is then given by 
\begin{equation}
    \MC_n(V):=\MC\left(\hom_{R}\left(N_*(\Delta^n,V\right)\right).
\end{equation}
The face and degeneracy maps are the maps induced by the face and degeneracy maps of $\{\Delta^n\}_{n \geq 0}$. 

\begin{lemma}
   Under the earlier assumptions from this section, the induced face maps $d_i:\MC_n(V)\rightarrow \MC_{n-1}(V)$ and $s_j:\MC(V)\rightarrow \MC_{n+1}(V)$ preserve Maurer-Cartan elements.
\end{lemma}

This follows from the fact that $\MC$ is a bifunctor in the coalgebra and the algebra. Since the cosimplicial maps are coalgebra maps they preserve the Maurer-Cartan equation. The Maurer-Cartan simplicial set therefore becomes indeed a well defined simplicial set.

\begin{remark}
    In characteristic $0$, the convolution algebra becomes canonically an $L_\infty$-algebra with the $L_\infty$-structure defined in \cite{Wierstra1}. This is unfortunately not the case when working over rings that do not contain $\Q$ as a subring. 
\end{remark}

\begin{remark}
    In our definition of the Maurer-Cartan simplicial set we have chosen to use the chains on the standard simplex. But this construction could of course also be done for other models of the simplex. It is a natural question to ask whether other models would give    homotopy equivalent Maurer-Cartan simplicial sets. It seems highly likely that the methods of Milham and Rogers (see  \cite{MilhamRogers}) would also apply to this more general setting, but this is beyond the scope of this paper.
\end{remark}


\subsection{The Maurer-Cartan simplicial set is a Kan complex}

In this section, we prove that the Maurer-Cartan simplicial set is a Kan complex.

\begin{theorem}\label{thrm:MCisKan}
    Let $V$ be a complete $\Omega \C$-algebra then the Maurer-Cartan simplicial set $\MC_\bullet(V)$ is a Kan complex.
\end{theorem}

To prove the theorem, we follow Getzler's proof of the fact that the Maurer-Cartan simplicial set or Deligne-Getzler-Hinich groupoid associated to an $L_\infty$-algebra (see \cite{Getzler}, Section 4). The main idea behind Getzler's proof is that the chains on $\Delta^n$ come with a retraction onto each of the vertices of $\Delta^n$. Given a horn $\Lambda^n_k$ in $V$, we can use this contraction to inductively build a horn filler. In \cite{deKleijnWierstra}, we generalized Getzler's proof to the case of $A_\infty$-algebras by using the cochains on the simplices. In the case of a general $\Omega \C$-algebra, we need to replace the cochains by the chains and replace the tensor product by the mapping space. The biggest difference is that in the $A_\infty$- and $L_\infty$-case the tensor product with the cochains (resp. polynomial de Rham forms) is again an $A_\infty$ (resp. $L_\infty$-algebra).  In our case, the tensor product is replaced by the convolution algebra which is not naturally an $\Omega \C$-algebra. We therefore need to work with the more general Maurer-Cartan equation from Equation \ref{eq:relativeMCequation}, which has as a consequence that our formulas are slightly different from Getzler's proof in \cite{Getzler}.

Before we prove Theorem \ref{thrm:MCisKan}, we first need a contraction on the chains of the simplex $\Delta^n$. Since all the arguments are completely analogues to the arguments in \cite{deKleijnWierstra}, we have left the proofs to the reader. For what follows we use the following notation. The $k$-dimensional subsimplex of $\Delta^n$ with vertices $i_0,...,i_k$ is denoted by $e_{i_0,...,i_k}$, with $0 \leq i_0 < ...< i_k\leq n$. A set of generators for $N_k(\Delta^n)$ as a chain complex is then given by $\{e_I\}$ where $I\subseteq \{0,...,n\}$ runs  over all subsets of order $k+1$. The differential is then given by 
\[
d(e_{i_0,...,i_k})=\sum_{j=0}^k (-1)^j e_{i_0,...,\widehat{i_j},...,i_k}, 
\]
where $\widehat{i_j}$ means that we omit the $i_j$th index.

The counit of $N_*(\Delta^n)$ is the map induced by the map of simplicial sets $\Delta^n\rightarrow \Delta^0$. Explicitly the counit $\epsilon:N_*(\Delta^n)\rightarrow R$  is defined on generators by $\epsilon(e_k)=\1$ and zero otherwise. The inclusion of the $k$th vertex is denoted by $p^k_n:\Delta^0 \rightarrow \Delta^n$ and is given by $p^k_n(e_0)=e_k$.

The composition $p_n^k\circ \epsilon:N_*(\Delta^n)\rightarrow N_*(\Delta^n)$ is homotopic to the identity $\id:N_*(\Delta^n)\rightarrow N_*(\Delta^n)$ via the chain homotopy $h^k_n:N_*(\Delta^n)\rightarrow N_{*+1}(\Delta^n)$ given by
\[
h^k_n(e_I):=(-1)^{s}e_{I\cup k},
\]
with $e_I\in N_*(\Delta^n)$ and where $I\cup k$ is defined as $0$ if $k$ was already an element of $I$. The sign $s$ is given by the number of elements in $I$ smaller than $k$.

The map $h^k_n$ is a chain homotopy between $p_n^k$ and the identity, i.e. it satisfies the following equation
\begin{equation}\label{eq:homotopy}
d h^k_n +h^k_n d = \id_{N_*(\Delta^n)}-p_n^k.    
\end{equation}

The map $p^k_n$ induces a map 
\[
\tilde{P_n^k} :\hom_{R}(N_*(\Delta^n),V)_d \rightarrow \hom_{R}(N_*(\Delta^0),V)_d\cong V
\]
given by
\[
\tilde{P_n^k} (\varphi):=\varphi \circ p^k_n
\]
with $\varphi \in\hom_{R}(N_*(\Delta^n),V)_d$. Similarly, the counit $\epsilon:N_*(\Delta^n)\rightarrow R$ induces a map
\[
E: V\cong \hom_{R}(N_*(\Delta^0),V)_d\rightarrow \hom_{R}(N_*(\Delta^n),V)_d
\]
given by
\[
E(\phi):=\phi \circ \epsilon
\]
with $\phi \in \hom_{R}(N_*(\Delta^0),V)_d$. The map 
\[
P^k_n:\hom_{R}(N_*(\Delta^n),V)_d\rightarrow \hom_{R}(N_*(\Delta^n),V)_d
\]
is defined by 
\[
P_n^k:=E \circ \tilde{P^k_n}.
\]
The homotopy $h^k_n$ induces a similar homotopy on the level of convolution algebras
\[
H^k_n:\hom_{R}(N_*(\Delta^n),V)_d\rightarrow \hom_{R}(N_*(\Delta^n),V)_{d+1}.
\]
given by 
\[
H^k_n(\varphi):=\varphi \circ h^k_n
\]
with $\varphi\in \hom_{R}(N_*(\Delta^n),V)_d$. The maps $\partial$, $E$, $P_n^k$ and $H^k_n$ satisfy the following identity
\begin{equation}\label{eq:chainhomotopy}
\partial H^k_n +H^k_n\partial = \id_{\hom_{R}(N_*(\Delta^n),V)}-P_n^k.
\end{equation}
We further define the map 
\[
R_n^k:\hom_{R}(N_*(\Delta^n),V)_d\rightarrow\hom_{R}(N_*(\Delta^n),V)_d
\]
as
\[
R^k_n:=\partial \circ H^k_n.
\]
Using these maps we can now prove Theorem \ref{thrm:MCisKan}.


\begin{proof}[Proof of Theorem \ref{thrm:MCisKan}]
    As mentioned earlier, we use a variation of Getzler's proof in which we start with a certain element in $\hom_{R}(N_*(\Delta^n),V)$ and extend this to a Maurer-Cartan element. 

    Suppose that we have a $k$-horn in $\MC_\bullet(V)$, i.e. a map $\varphi:N_*(\Lambda^n_k)\rightarrow V$, then we need to construct a horn filler, i.e. a map $\psi:N_*(\Delta^n)\rightarrow V$ that fills $\varphi$.  We define $\psi$ inductively and start with $\psi_1$ 
    which we define as $\psi_1=\xi+\rho$, where
    \[
    \xi(e_k):=\varphi(e_k)
    \]
    and zero otherwise and 
    \[
    \rho(e_I) := 
	 \begin{cases} \varphi(e_I) &\mbox{if } e_I \neq e_k,e_{0...\hat{k}...n} \mbox{ or } e_{01...n}, \\
	 	\sum_{i\neq k}\varphi(e_{0...\hat{i}...n}) &  \mbox{if } e_I=e_{0...\hat{k}...n}, \\
   0 & \mbox{if } e_I=e_k \mbox{ or } e_{01...n}.
 	\end{cases}
    \]
    Since $\varphi$ is a Maurer-Cartan element, the element $\xi$ is also a Maurer-Cartan element. It is further straightforward to see that $\partial \rho=0$, so it is a cycle.  
    

    In most cases, the element $\psi_1$ is not a Maurer-Cartan element. Since $\xi$ is a Maurer-Cartan element and $\rho$ a cycle, it does however satisfy the Maurer-Cartan equation in 
    \[
    F_1\hom(N_*(\Delta^n),V)/F_2\hom(N_*(\Delta^n),V),
    \]
    i.e.  it satisfies the Maurer-Cartan equation modulo elements of filtration degree $2$. We proceed by adding a "correction" term $\gamma_1$ that defines an element $\psi_2$ which satisfies the Maurer-Cartan equation up to elements of filtration degree $3$. First, we define  
    \[
    \gamma_1 := H^k_n(\partial \psi_1 +\star_{\iota} (\psi_1)).
    \]
    The element $\gamma_1$ is of filtration degree greater or equal $2$ because both $\star_\iota(\psi_1)$ and $\partial \psi_1$ are  of filtration degree $\geq 2$. We define $\psi_2$ as 
    \[
    \psi_2:=\psi_1-\gamma_1.
    \]
    Next we show that the element $\psi_2$ satisfies the Maurer-Cartan equation up to filtration degree $3$. We have
    \begin{equation}\label{eq:psi_2isMC}
    \partial(\psi_2)+\star_{\iota}(\psi_2)= \partial \psi_1 -\partial \gamma_1 + \star_\iota(\psi_1-\gamma_1).
    \end{equation}
    Since $\gamma_1$ is of filtration degree $\geq 2$, the element $\star_\iota(\psi_1-\gamma_1)$ can be rewritten as
    \[
    \star_\iota(\psi_1-\gamma_1)=\star_\iota(\psi_1)+\mbox{terms of filtration degree $\geq 3$}.
    \]
    So modulo elements of filtration degree $\geq 3$, Equation \ref{eq:psi_2isMC} reduces to
    \begin{equation}\label{eq:partialpsi1}
        \partial \psi_1 -\partial \gamma_1 + \star_\iota(\psi_1).
    \end{equation}
    If we apply Equation \ref{eq:chainhomotopy} to $\partial \gamma_1$, we get
    \begin{align}
    \partial \gamma_1 & =\partial H^k_n(\partial \psi_1 +\star_{\iota} \psi_1) \\
    &= \partial \psi_1 +\star_\iota(\psi_1) -P_n^k(\partial(\psi_1))-P_n^k(\star_\iota(\psi_1)) - H^k_n (\partial^2 \psi_1) - H_n^k(\partial ( \star_\iota(\psi_1))).  \end{align}
    So if we combine this with Equation \ref{eq:partialpsi1}, we get
    \begin{equation}
        -P_n^k(\partial(\psi_1))-P_n^k(\star_\iota(\psi_1)) - H^k_n (\partial^2 \psi_1) - H_n^k(\partial ( \star_\iota(\psi_1))).
    \end{equation}
    Because $\xi$ satisfies the Maurer-Cartan equation, the terms $-P_n^k(\partial(\psi_1))-P_n^k(\star_\iota(\psi_1))$ are zero. The term $- H^k_n (\partial^2 \psi_1)$ is also zero since it involves $\partial^2$. So we are left with the term $-H_n^k(\partial ( \star_\iota(\psi_1)))$ and we need to show that this is of filtration degree $\geq 3$. If we use the Leibniz rule for algebras over operads we get 
    \[
    -H_n^k(\partial ( \star_\iota(\psi_1)))=\sum_{r\geq 2} \partial(\mu_r)(\psi_1^{\otimes r}) +\sum_{r \geq 2, 0\leq l \leq r} \mu_r(\psi_1^{\otimes l} \otimes \partial(\psi_1) \otimes \psi_1^{\otimes r-l-1}). 
    \]
    Since $\partial \psi_1$ is of filtration degree $\geq 2$, all the terms of the form $\mu_r(\psi_1^{\otimes l} \otimes \partial(\psi_1) \otimes \psi_1^{\otimes r-l-1})$ are of filtration degree $\geq 3$. Further, for $r \geq 3$, the terms $\partial(\mu_r)(\psi_1^{\otimes r})$ are all of filtration degree $\geq 3$. The term $\partial(\mu_2)(\psi_1\otimes \psi_1)=0$ because $\mu_2$ is the arity two component of an operadic twisting morphism. The element $\psi_2$ is therefore Maurer-Cartan up to terms of filtration degree $3$. 

    We continue inductively by defining the next "correction" terms as
    \[
    \gamma_i:=H_n^k\left(\partial \psi_i+\star_{\iota}(\psi_i)\right)
    \]
and 
\[
\psi_{i+1}=\psi_i-\gamma_i.
\]
By exactly the same arguments it follows that $\psi_i$ is Maurer-Cartan up to elements of filtration degree $\geq i+1$. The element $\psi$ is then defined as
\[
\psi:=\lim \psi_i
\]
by completeness this limit actually converges and is a Maurer-Cartan element. It can further be shown that it satisfies the properties of a horn filler for $\varphi$ and therefore proves that $\MC_\bullet(V)$ is a Kan complex.


\end{proof}

\section{Comparison to other approaches and examples}\label{sec:examples}

Recently, in \cite{DSV} another approach to the twisting procedure was described by using the gauge group. It is currently unclear how to their twisting procedure exactly compares to ours.  In both cases, the unitary operation plays an essential role but it is not clear how the unital Hopf (co)operad condition compares to their unital extendability condition. However, the approach in this paper has multiple advantages compared to \cite{DSV}. First of all, our constructions work over arbitrary rings and not just fields of characteristic $0$ or rings that contain $\Q$. Second, the constructions of this paper also apply to differential graded operads and not just operads defined by quadratic data.

We finish this paper by showing that the Koszul dual of  every unitary operad $\P$ in simplicial sets admits a twisting procedure. By unitary in simplicial sets we mean $\P(0)=*$.  We further show that one can construct a Maurer-Cartan simlicial set for the Barratt-Eccles operad and its $\E_n$-suboperads.

\begin{theorem}
    Let $\P$  be a unitary operad in simplicial sets with finitely many non-degenerate simplices in each arity, then $N^*(\P)$ satisfies the conditions of Theorem \ref{thrm:twistingprocedure} and therefore admits a twisting procedure.
\end{theorem}

\begin{proof}
We need to show that $N^*(\P)$ is a unital Hopf cooperad and that there is a map of cooperads $u\Cocom\rightarrow N^*(\P)$. Since $\P$ is an operad of finite type, the cochains on $\P$ are naturally a cooperad. The Hopf structure comes from the chain level cup product, it follows from the properties of the cochain functor that the cup product is unital and is compatible with the cooperad structure. Since $u\Com$ is the terminal operad in simplicial sets, every operad $\P$ comes equipped with a unique map $\P\rightarrow u\Com$. The induced map on cochains is the map    $u\Cocom\rightarrow N^*(\P)$ we need.
\end{proof}

The main example of a class of cooperads that admit a Maurer-Cartan simplicial set are the $\E_n$-subcooperads of the dual Barratt-Eccles operad. The Barratt-Eccles operad $\BE_\infty$ is an operad in simplicial sets which naturally acts on the chains and cochains of a simplicial set (see \cite{BergerFresse} for more details). It comes with a sequence of suboperads 
\begin{equation}\label{eq:sequence}
\BE_1 \hookrightarrow \BE_2 \hookrightarrow \BE_3 \hookrightarrow ... \BE_\infty,
\end{equation}
where each $\BE_n$ models the chains on the little $n$-disks operad. We denote the normalized cochains on $\BE_n$ by $\E_n:=N^*(\BE_n)$. Since the cochains are contravariant, the sequence of maps from Equation \ref{eq:sequence} induces a sequence of maps of cooperads
\[
\E_\infty\rightarrow ... \E_3 \rightarrow \E_2 \rightarrow \E_1.
\]
So every $\E_n$-cooperad has a map $\E_\infty\rightarrow \E_n$ and therefore satisfies the conditions from Section \ref{sec:MCsimplicialsets}.  The Koszul dual of $\E_n$ is given by $\Omega \E_n$ and Fresse showed in \cite{FresseEn} that $\Omega \E_n$ is weakly equivalent to $\Lambda^{-n} \E_n^{\vee}$, where $\Lambda^{-n}$ denotes the operadic desuspension of $\E_n$ (note that we exchanged $\E_n$ and $\E_n^\vee$ from \cite{FresseEn}). By the results of Lurie from \cite{Lurie2010}, it turns out that these Maurer-Cartan simplicial sets control $\E_n$ deformation problems.

\bibliographystyle{plain} 
\bibliography{bibliography}

\end{document}